\documentclass[12pt]{amsart}
\usepackage{amscd,amssymb,amsthm,verbatim}
\newcommand{\Aut}{\operatorname{Aut}}

\newcommand{\C}{{\mathbb C}}

\newcommand{\const}{\operatorname{const.}}

\newcommand{\diam}{\operatorname{diam}}

\newcommand{\dvol}{\operatorname{dvol}}

\newcommand{\GL}{\operatorname{GL}}

\newcommand{\HH}{\operatorname{H}}

\newcommand{\Id}{\operatorname{Id}}

\newcommand{\N}{{\mathbb N}}
\newcommand{\nil}{\operatorname{nil}}
\newcommand{\Nil}{\operatorname{Nil}}

\newcommand{\OO}{\operatorname{O}}

\newcommand{\R}{{\mathbb R}}

\newcommand{\Ric}{\operatorname{Ric}}

\newcommand{\Rm}{\operatorname{Rm}}

\newcommand{\SL}{\operatorname{SL}}
\newcommand{\SO}{\operatorname{SO}}

\newcommand{\Tr}{\operatorname{Tr}}

\newcommand{\vol}{\operatorname{vol}}
\newcommand{\Z}{{\mathbb Z}}

\numberwithin{equation}{section}
\theoremstyle{plain}

\newtheorem{theorem}[equation]{Theorem}
\newtheorem{proposition}[equation]{Proposition}
\newtheorem{corollary}[equation]{Corollary}

\theoremstyle{remark}

\newtheorem{remark}[equation]{Remark}
\newtheorem{example}[equation]{Example}
\usepackage{graphicx}
\input{amssym.def}
\input{amssym}
\input{epsf}
\usepackage{a4wide}

\begin{document}

\title[Collapsing geometry]
{The collapsing geometry of almost Ricci-flat $4$-manifolds}

\author{John Lott}
\address{Department of Mathematics\\
University of California, Berkeley\\
Berkeley, CA  94720-3840\\
USA} \email{lott@berkeley.edu}

\thanks{Research partially supported by NSF grant
DMS-1510192}
\date{March 20, 2020}
\begin{abstract}
  We consider Riemannian
  $4$-manifolds that Gromov-Hausdorff converge to
  a lower dimensional limit space, with the Ricci tensor going to zero.
  Among other things, we show that
  if the limit space is
  two dimensional then under some mild assumptions, the
  limiting four dimensional geometry away from the curvature blowup region is
  semiflat K\"ahler.
\end{abstract}

\maketitle


\section{Introduction} \label{sect1}

When considering Einstein manifolds, or almost Einstein manifolds, the
four dimensional case is especially interesting.  This paper is about
almost Ricci-flat $4$-manifolds, meaning compact
$4$-manifolds $M$ that admit a sequence of Riemannian metrics
$\{g_i\}_{i=1}^\infty$ with
$\lim_{i \rightarrow \infty} \| \Ric(M, g_i) \|_\infty
\cdot \diam(M, g_i)^2 = 0$.
Special cases come from 
Ricci-flat $4$-manifolds.
The known examples of the latter are finitely covered by a flat
torus or by a Ricci-flat K\"ahler metric on a $K3$ manifold.
There are almost Ricci-flat $4$-manifolds
that do not admit Ricci-flat metrics \cite{Anderson (1992)}.

Fixing an upper diameter bound for $\{(M, g_i)\}_{i=1}^\infty$,
one can divide the study of almost Ricci-flat $4$-manifolds into the
noncollapsed case, where there is a definite positive lower volume bound, and
the collapsing case, where the volume goes to zero.
In the noncollapsed case, a Gromov-Hausdorff limit
(as the Ricci curvature goes to zero) is a four dimensional Ricci-flat
orbifold  with isolated orbifold points, as follows from
work of Anderson \cite{Anderson (1990)}, Bando-Kasue-Nakajima
\cite{Bando-Kasue-Nakajima (1989)} and Tian
\cite{Tian (1990)}.
The orbifold points are caused by
noncompact Ricci-flat ALE manifolds (or orbifolds) that bubble off.
There is a bubble tree description of the sequence
\cite{Anderson-Cheeger (1991),Bando (1990)}.
In \cite{Kapovitch-Lott (2016)},
sufficient topological conditions were given 
for a noncollapsed almost Ricci-flat $4$-manifold to admit a Ricci-flat
metric.
There are probably also noncollapsed almost Ricci-flat $4$-manifolds that do
not admit Ricci-flat metrics
\cite{Brendle-Kapouleas (2017)}.

In the collapsing case, fundamental work was done by Cheeger and Tian
\cite{Cheeger-Tian (2006)}. Allowing the manifolds to vary,
 let $\{(M_i, g_i)\}_{i=1}^\infty$ be a sequence of 
  compact connected orientable
  Riemannian $4$-manifolds so that for some $C \in \N$ and $D < \infty$,
  \begin{itemize}
    \item $\chi(M_i) \le C$ for all $i$,
\item $\diam(M_i, g_i) \le D$ for all $i$, 
\item $\lim_{i \rightarrow \infty} \|\Ric(M_i, g_i)\|_\infty = 0$ and
  \item $\lim_{i \rightarrow \infty} \vol(M_i, g_i) = 0$.
  \end{itemize}
  After passing to a subsequence, we can assume that
  $\lim_{i \rightarrow \infty} (M_i, g_i) = (X, d_X)$ in the Gromov-Hausdorff
  topology, for some compact
  metric space $X$ whose Hausdorff dimension is less than four.
  As we will review in Subsection \ref{subsect3.1},
  Cheeger and Tian
 showed that for large $i$,
  each $(M_i, g_i)$
  has a small ``curvature blowup''
  region where the curvature concentrates in an $L^2$-sense,
  and a ``regular'' region with {\it a priori} curvature bounds.
    We can assume that the regular regions converge to a subset
  $X_{reg} \subset X$, whose complement in $X$ is a finite set. In particular,
  if $B$ is a connected component of $X_{reg}$ then
  taking the metric completion of $B$ amounts to adding a finite number of points.

  We are interested in the four dimensional geometry of the regular regions.
  From work of Cheeger, Fukaya and Gromov, culminating in
  \cite{Cheeger-Fukaya-Gromov (1992)}, collapsing
  regions with bounded curvature acquire continuous symmetries in the limit.
  (The results from \cite{Cheeger-Fukaya-Gromov (1992)} were localized in
  \cite[Section 2]{Cheeger-Tian (2006)}.)
  A convenient language to formalize the collapsing
  limit, with its symmetries, is that of
  Riemannian groupoids, as described in \cite[Section 5]{Lott (2007)}.
  A brief introduction to the use of Riemannian groupoids in collapsing theory
  is in \cite[Section 3]{Lott (2010)}.
  Passing to a subsequence, we can assume
  that the regular regions, approaching $B$, also converge
  in the sense of Riemannian groupoids, to
  a four dimensional smooth Ricci-flat
  Riemannian groupoid ${\mathcal X}$ whose orbit
  space is $B$.

  To state the main result of this paper, we recall that
  a (possibly incomplete) connected
Riemannian manifold is {\em parabolic}
if any 
$C^2$-regular function $f$ that is bounded above, and satisfies
$\triangle f \ge 0$, is constant. (If the manifold has boundary then we
require $f$ to vanish on the boundary.)
Some equivalent conditions for parabolicity are given in
\cite[Theorem 5.1]{Grigor'yan (1999)}.
There is a similar definition for Riemannian orbifolds.

For example,
the complement of a finite number of points in a closed Riemannian manifold,
of dimension greater than one,
is parabolic \cite[Corollary 5.4]{Grigor'yan (1999)}.
Whether or not a two dimensional Riemannian manifold is parabolic
only depends on its underlying conformal structure.

Let $\sqrt{\det G}$ denote the
relative volume function of the orbits of ${\mathcal X}$,
a function on $B$.
(For example, if ${\mathcal X}$ comes from a free torus action
then $\sqrt{\det G}$ describes the relative volumes
of the torus orbits.)

\begin{theorem} \label{thm1.1}
  \begin{enumerate}
  \item If $\dim(B) = 3$ then $B$ is an orbifold.  If $B$ is parabolic
    and $\sqrt{\det G}$ is bounded above then $B$ is flat and
    the $4$-dimensional geometry of ${\mathcal X}$ is flat.
  \item If $\dim(B) = 2$ then $B$ is an orbifold-with-boundary.
    If $B$ is parabolic and $\sqrt{\det G}$ is bounded above
    then $B$ is boundaryless with nonnegative scalar curvature,
    and the $4$-dimensional geometry of ${\mathcal X}$ is
    semiflat K\"ahler.
  \item If $\dim(B) = 1$ then $B$ is a circle or an interval.
    The $4$-dimensional
    geometry of ${\mathcal X}$, over the interior of $B$, is flat,
    or a Riemannian Kasner geometry, or a
    Riemannian Bianchi-II geometry.
  \item If $\dim(B) = 0$ then $B$ is a point and
    the $4$-dimensional geometry of ${\mathcal X}$ is flat.
    \end{enumerate}
  \end{theorem}

The notion of a semiflat K\"ahler metric is given in
\cite[Section 3]{Gross-Tosatti-Zhang (2013)},
\cite[Section 3.2]{Hein (2012)}, \cite{Hein-Tosatti (2015)}
and
\cite[Section 3.1]{Tosatti-Zhang (2017)}, among other places.
The Riemannian Kasner geometry and the
Riemannian Bianchi-II geometry are defined in
Subsection \ref{subsect3.4}.
Appropriate sequences of
Ricci-flat $K3$ manifolds give examples of parts
(1), (2) and (3) of Theorem \ref{thm1.1}; see Examples
\ref{ex3.5}, \ref{ex3.8} and \ref{addedexample}.
The constructions in \cite{Anderson (1992)} give further
examples.

Theorem \ref{thm1.1} can be viewed in two ways.
On the one hand, it gives some explanation for the geometry of the
regular regions
seen in the known almost Ricci-flat examples, and indicates what
other examples may exist.  On the other hand, it
shows what restrictions would have to be lifted in order to find
exotic examples.

In the setting of Theorem \ref{thm1.1}, if $B$ is not a point then
the local symmetry Lie algebra
of ${\mathcal X}$ must be $\nil^3$ or $\R^N$, where $1 \le N \le 3$.
The $\nil^3$ case can be handled separately, so the main task in
proving Theorem \ref{thm1.1} is to analyze the Ricci-flat equations
on a manifold with a local $\R^N$-symmetry.

More generally, we look at Einstein manifolds of arbitrary dimension
with a local $\R^N$-symmetry. In the case of a locally free action,
the Riemannian metric gives a
distribution that is transverse to the local orbits of the $\R^N$-action.
One interesting feature is that when
the quotient space is two dimensional, there are conserved quantities that,
under suitable topological conditions, force the 
distribution to be integrable; see Corollary \ref{cor2.16}.

\subsection{Earlier work}

Einstein manifolds with symmetries have been considered in many
papers, including
\cite{Calderbank-Pedersen (2002),Dammerman (2004),Wang (2012)}.

In \cite{Naber-Tian}, Naber and Tian looked at collapsing sequences
of manifolds having bounded diameter and bounded curvature, with
the Ricci tensor going to zero.  (In the four dimensional case,
this corresponds to not having any curvature blowup regions.)  They showed
that the Gromov-Hausdorff limit is a Ricci-flat orbifold.
Instead of Riemannian groupoids, they used a notion of
$N^*$-bundles.  The argument used a formula for
$\triangle \ln \det G$, along with the maximum
principle; compare with (\ref{2.9}).

In work in progress, Cheeger and Tian use the framework of
\cite{Cheeger-Tian (2006)} to study
finite-volume complete Einstein $4$-manifolds with negative
Einstein constant and ends that are asymptotic to rays.
They use the collapsing structure at infinity to identify the possible 
model geometries (real or complex hyperbolic cusps) and show that
along an end, a model
geometry is indeed asymptotically approached.

\subsection{Structure of the paper}

In Section \ref{sect2} we analyze Ricci-flatness for the total space
of a (twisted) principal bundle with abelian structure group.
The proof of Theorem \ref{thm1.1} is given in Section \ref{sect3}.
In fact, we prove the conclusions of
Theorem \ref{thm1.1} when the upper diameter
bound is replaced by an upper volume bound, and the
volume is only assumed to go to zero in the local sense of
(\ref{neweqn}). With these more general assumptions, we
have to introduce basepoints, which is why we only discuss
the bounded diameter case in this introduction.

More detailed descriptions are given at the beginnings of the sections.

I thank Jeff Cheeger, Hans-Joachim Hein, Claude LeBrun,
Gang Tian and Jeff Viaclovsky for helpful
comments. I also thank the referee for pointing out a mistake in an earlier
version of the paper.

\section{Isometric free local $\R^N$-actions} \label{sect2}

In this section we consider a (twisted) principal bundle
with abelian structure group and an adapted Riemannian metric.
In Subsection \ref{subsect2.1} we define the relevant bundles and
metrics, and give the formula for
the Ricci curvature of the total space. In Subsection \ref{subsect2.2}
we show how the Ricci-flat condition simplifies when the base is parabolic and
the fiberwise volume forms are relatively bounded.  Subsection
\ref{subsect2.3} gives the consequences when the fibers are
one dimensional.  In Subsection \ref{subsect2.4} we discuss the
conserved quantities that arise when the base is two dimensional,
and show what the results of Subsection \ref{subsect2.2} become in the
two dimensional case. Finally,
Subsection \ref{subsect2.5} is about a one dimensional base.

The results of this section extend directly to the case when the
base is an orbifold. We use the Einstein summation convention freely.

\subsection{Ricci curvature equation} \label{subsect2.1}

Let ${\mathcal G}$ be a Lie group, with Lie algebra ${\frak g}$.
Let $\Aut({\mathcal G})_\delta$ denote
the automorphism group of ${\mathcal G}$, with the discrete topology, and let
${\mathcal G} \rtimes \Aut({\mathcal G})_\delta$ denote
the semidirect product.
If $B$ is a connected smooth manifold, let $P \rightarrow B$ be
a principal ${\mathcal G} \rtimes \Aut({\mathcal G})_\delta$-bundle.
From the homomorphism
${\mathcal G} \rtimes \Aut({\mathcal G})_\delta \rightarrow
\Aut({\mathcal G})_\delta$,
there is a corresponding
principal bundle 
${\frak F}$ on $B$ with discrete structure group $\Aut({\mathcal G})_\delta$.
From the action of ${\mathcal G} \rtimes \Aut({\mathcal G})_\delta$
on ${\mathcal G}$,
there is also
a fiber bundle ${\frak E}$ on $B$ associated to $P \rightarrow B$,
with fiber ${\mathcal G}$. 
One can think of ${\frak E}$ as an ${\frak F}$-twisted
${\mathcal G}$-principal bundle, in the sense that ${\frak E}$ has
free local ${\mathcal G}$-actions that are globally twisted by
${\frak F}$.
In addition,
there is an flat vector bundle $e$ on $B$ associated to ${\frak F}$,
  with fiber ${\frak g}$. 

  In what follows, we will be interested in the case
  when ${\mathcal G}$ is an $N$-dimensional connected
  abelian Lie group with $N \ge 1$.
(An example is when ${\mathcal G} = T^N$, ${\frak F}$ is a trivial principal
$\GL(N, \Z)$-bundle
and ${\frak E}$ is a principal $T^N$-bundle on $B$.)
Let $M$ be the total space of 
${\frak E}$. We write $\dim(B) = n$ and $\dim(M) = m = N+n$.

Let
$\overline{g}$ be a Riemannian metric on $M$ with a
free local isometric ${\mathcal G}$-action (globally twisted by ${\frak F}$).
In adapted local coordinates, we can write
\begin{equation} \label{2.1} 
\overline{g} \: = \: \sum_{I,J=1}^N G_{IJ} \: (dx^I + A^I) (dx^J + A^J) \: + \:
\sum_{\alpha, \beta = 1}^{n} g_{\alpha \beta} \: db^\alpha db^\beta.
\end{equation}
Here the $x^I$'s are linear local coordinates on the fibers of
$M \rightarrow B$,
$(G_{IJ})$ is the local expression of a Euclidean inner product on
$e$,
$\sum_{\alpha, \beta = 1}^{n} g_{\alpha \beta} \: db^\alpha db^\beta$ is
the local expression of a metric $g_B$ on $B$ and
$A^I = \sum_{\alpha} A^I_\alpha db^\alpha$ are the components of
a local $e$-valued $1$-form describing an connection $A$ on the
twisted ${\mathcal G}$-bundle $M \rightarrow B$.

Put $F^I_{\alpha \beta} = \partial_\alpha A^I_\beta - 
\partial_\beta A^I_\alpha$.
At a given point $b \in B$, we can assume that $A^I(b) = 0$.
We write
\begin{equation} \label{2.2}
G_{IJ;\alpha \beta} \: = \: G_{IJ,\alpha \beta} \: - \:
\Gamma^{\sigma}_{\: \: \alpha \beta} \: G_{IJ, \sigma},
\end{equation}
where $\{\Gamma^{\sigma}_{\: \: \alpha \beta}\}$ are the
Christoffel symbols for the metric $g_{\alpha \beta}$ on $B$.

From \cite[Section 4.2]{Lott (2010)},
the Ricci tensor of ${\overline g}$ on $M$
is given in terms of the curvature tensor
$R_{\alpha \beta \gamma \delta}$ of $B$, the $2$-forms $F^I_{\alpha \beta}$
and the metrics $G_{IJ}$ by
\begin{align} \label{2.3} 
\overline{R}_{IJ}^{\overline{g}} 
\:  =  \: & - \: \frac12 \: g^{\alpha \beta} \: 
G_{IJ; \alpha \beta} \: - \: \frac14 \: g^{\alpha \beta} \:
G^{KL} \: G_{KL, \alpha} \: G_{IJ, \beta} \: + \:
\frac12 \: g^{\alpha \beta} \: G^{KL} \: G_{IK, \alpha} \:
G_{LJ, \beta} \: + \\
&  \frac14 \: g^{\alpha \gamma} \: g^{\beta \delta} \:
G_{IK} \: G_{JL} \: F^K_{\alpha \beta} \: F^L_{\gamma \delta} \notag \\
\overline{R}_{I \alpha}^{\overline{g}}
 \:  =  \: & \frac12 \: g^{\gamma \delta} \:
G_{IK} \: F^K_{\alpha \gamma; \delta} \: + \: \frac12 \: g^{\gamma \delta} \:
G_{IK, \gamma} \: F^K_{\alpha \delta} \: + \: \frac14 \:
g^{\gamma \delta} \: G_{Im} \: G^{KL} \: G_{KL, \gamma} \: F^m_{\alpha \delta} 
\notag \\
\overline{R}_{\alpha \beta}^{\overline{g}} 
\:  =  \: & R_{\alpha \beta}^g \: - \: 
\frac12 \: G^{IJ} \: G_{IJ; \alpha \beta} \: + \: \frac14 \:
G^{IJ} \: G_{JK,\alpha} \: G^{KL} \: G_{LI,\beta} \: - \:
\frac12 \: g^{\gamma \delta} \: \: G_{IJ} \: F^I_{\alpha \gamma} \:
F^J_{\beta \delta}. \notag
\end{align}
The scalar curvature is
\begin{align} \label{2.4}
\overline{R}^{\overline{g}} \: = \: & 
R^g \: - \: g^{\alpha \beta} G^{IJ}  \: G_{IJ; \alpha \beta} \: + \: 
\frac34 \: g^{\alpha \beta} \: G^{IJ} \: G_{JK, \alpha} \: G^{KL} \:
G_{LI, \beta} \\
& \: - \: \frac14 \: g^{\alpha \beta} \: G^{IJ} \: 
G_{IJ, \alpha} \: G^{KL} \: G_{KL, \beta} \: - 
 \: \frac14 \:
g^{\alpha \gamma} \: g^{\beta \delta} \: G_{IJ} \:
F^I_{\alpha \beta} \: F^J_{\gamma \delta}. \notag
\end{align}

In the rest of this section we will assume that the flat vector bundle $e$ has
holonomy in $\det^{-1}(\pm 1) \subset \GL(N, \R)$,
so that $\det G$ is globally defined
on $B$.

\subsection{General characterization of Ricci-flat metrics} \label{subsect2.2}

Suppose that $\Ric(M, \overline{g}) = \lambda \overline{g}$ for some
$\lambda \in \R$. Then the first equation in (\ref{2.3}) gives
\begin{align} \label{2.5}
N\lambda = & - \: \frac12 \: g^{\alpha \beta} \: 
G^{IJ} \: G_{IJ; \alpha \beta} \: - \: \frac14 \: g^{\alpha \beta} \:
G^{KL} \: G_{KL, \alpha} \: G^{IJ} \: G_{IJ, \beta} \: + \\
& \frac12 \: g^{\alpha \beta} \: G^{JI} \: G_{IK, \alpha} \:  G^{KL} \: 
G_{LJ, \beta} \: + 
\frac14 \: |F|^2, \notag
\end{align}
where
\begin{equation} \label{2.6}
|F|^2 = g^{\alpha \gamma} \: g^{\beta \delta} \:
G_{IJ} \: F^I_{\alpha \beta} \: F^J_{\gamma \delta}.
\end{equation}
Now
\begin{equation} \label{2.7}
  \nabla_\alpha (\det G)^\frac12 = \frac12 (\det G)^{\frac12}
  G^{IJ} G_{IJ;\alpha}
\end{equation}
and 
\begin{align} \label{2.8}
&  \triangle (\det G)^{\frac12} =  (\det G)^{\frac12} \cdot \\
&  \left( \frac12 g^{\alpha \beta} G^{IJ} G_{IJ;\alpha \beta} -
  \frac12 g^{\alpha \beta} G^{JI} G_{IK;\alpha} G^{KL} G_{LJ;\beta} +
  \frac14 g^{\alpha \beta} G^{KL} G_{KL;\alpha} G^{IJ} G_{IJ;\beta} \right).
  \notag
\end{align}
Thus (\ref{2.5}) becomes
\begin{equation} \label{2.9}
  \triangle \sqrt{\det G} = \left( \frac14 |F|^2 - \lambda N
  \right) \sqrt{\det G}.
\end{equation}

Recall the notion of a parabolic Riemannian manifold from the introduction.
Information about parabolic Riemannian manifolds is in
\cite[Section 5]{Grigor'yan (1999)}.

\begin{proposition} \label{prop2.10}
  If
\begin{itemize}
\item  $B$ is parabolic,
\item $\lambda \le 0$ and
\item $\det G$ is bounded above
\end{itemize}
then
\begin{enumerate}
\item $\det G$ is constant,
\item $F = 0$,
\item $\lambda = 0$ and
\item $g^{\alpha \beta} \: 
  G_{IJ; \alpha \beta} \: - \:
   g^{\alpha \beta} \: G^{KL} \: G_{IK, \alpha} \:
G_{LJ, \beta} = 0$.
  \end{enumerate}
\end{proposition}
\begin{proof}
  From (\ref{2.9}), $\sqrt{\det G}$ is subharmonic. Since $B$ is
  parabolic and $\sqrt{\det G}$ is bounded above, it must be constant.
  Then the right-hand side of (\ref{2.9}) vanishes, which implies that
  $F = 0$ and $\lambda = 0$. Substituting this into the first equation of
  (\ref{2.3}), whose left-hand side vanishes, proves the proposition.
  \end{proof}

Normalizing $\det G$ to be one, the equation
\begin{equation} \label{2.11}
g^{\alpha \beta} \: 
  G_{IJ; \alpha \beta} \: - \:
   g^{\alpha \beta} \: G^{KL} \: G_{IK, \alpha} \:
   G_{LJ, \beta} = 0
\end{equation}
is the local expression for a harmonic map from $B$ to the symmetric
space $\det^{-1}(\pm 1)/\OO(N) \cong \SL(N,\R)/\SO(N)$
\cite[Proposition 4.17]{Lott (2007)}. More globally, fixing a basepoint
$b_0 \in B$, let
$\rho : \pi_1(B, b_0) \rightarrow \det^{-1}(\pm 1)$ be the monodromy of the
flat vector bundle $e$.
Then (\ref{2.11}) is the equation for a
$\rho$-equivariant harmonic map 
$\widetilde{G} : \widetilde{B} \rightarrow
\det^{-1}(\pm 1)/\OO(N) \cong \SL(N,\R)/\SO(N)$
on the universal cover $\widetilde{B}$.

\subsection{Codimension-one base} \label{subsect2.3}
 
Returning to the equations (\ref{2.3}), suppose that $N=1$.
The matrix $(G_{IJ})$ just becomes a function $G$.

\begin{proposition} \label{prop2.12}
  If
\begin{itemize}
\item  $B$ is parabolic,
\item  $\lambda \le 0$ and
\item  $G$ is uniformly bounded above
\end{itemize}
then $G$ is a constant function, $F = 0$, $\lambda = 0$ and
$B$ is Ricci-flat.
\end{proposition}
\begin{proof}
From Proposition \ref{prop2.10}, $G$ is constant, $F = 0$ and $\lambda = 0$. 
The third equation in (\ref{2.3}) implies that $B$ is Ricci-flat.
\end{proof}

\subsection{Two dimensional base} \label{subsect2.4}

Returning to the equations (\ref{2.3}), suppose that $n=2$ and
$(M, \overline{g})$ is Einstein.

\begin{proposition} \label{prop2.13}
  The second equation of (\ref{2.3}) is equivalent to the statement that
  $\sqrt{\det G} \: G_{IJ} \frac{F^J}{\dvol_B}$ defines a flat section of
  $e^*$.
\end{proposition}
\begin{proof}
  Choose local isothermal coordinates on $B$ so that
  $g = e^{2\phi} \left( (dx^1)^2 + (dx^2)^2 \right)$.
  The nonzero Christoffel symbols, up to symmetries, are
  \begin{align} \label{2.14}
    & \Gamma^1_{\: \: 11} = \Gamma^2_{\: \: 21} = - \Gamma^1_{\: \: 22} =
    \partial_{x^1} \phi, \\
    & \Gamma^2_{\: \: 22} = \Gamma^1_{\: \: 12} = - \Gamma^2_{\: \: 11} =
    \partial_{x^2} \phi. \notag
    \end{align}
  The second equation of (\ref{2.3}) becomes equivalent to
  \begin{equation} \label{2.15}
    \partial_\alpha
    \left( \sqrt{\det G} \: G_{IJ} e^{-2 \phi} F^J_{12} \right) = 0.
    \end{equation}
  As $\dvol_B = e^{2\phi} dx^1 \wedge dx^2$, we can express (\ref{2.15}) in more
  invariant terms as saying that
  $\sqrt{\det G} \: G_{IJ} \frac{F^J}{\dvol_B}$ is locally constant.
  By assumption, $\det e$ is a trivial bundle, so 
  $\sqrt{\det G} \: G_{IJ} \frac{F^J}{\dvol_B}$ is naturally a section
  of $e^*$. Equation (\ref{2.15}) says that it is a locally constant section.
\end{proof}

\begin{corollary} \label{cor2.16}
  If the monodromy representation $\rho : \pi_1(B, b_0) \rightarrow
  \det^{-1}(\pm 1) \subset \GL(N, \R)$ of $e^*$ does not have a
  trivial one-dimensional subrepresentation
  then $F = 0$.
  \end{corollary}

\begin{example}
  In the collapsing of a $K3$ manifold considered in
  \cite{Gross-Wilson (2000)}, the base $B$ is $S^2$ minus
  24 points; see Example \ref{ex3.8} below. The vector
  bundle $e^*$ has unipotent holonomy when going around
  a small loop around any of the 24 punctures.  However,
  the invariant subspaces do not line up globally and
  there is no nonzero flat section of $e^*$. Hence $F=0$,
  as seen directly in \cite{Gross-Wilson (2000)}.
  \end{example}

\begin{remark}
  In the Lorentzian setting, the expression
  $\sqrt{\det G} \: G_{IJ} \frac{F^J}{\dvol_B}$ coincides with the
  ``twist constants'' of general relativity, although in the
  relativity literature the relation to curvature seems to be
  missing, along with the topological meaning.
    \end{remark}

We note that parabolicity of a two dimensional 
Riemannian manifold just
depends on the underlying conformal structure.

\begin{proposition} \label{prop2.17}
  Suppose that $\dim(B) = 2$ and that the hypotheses of Proposition \ref{prop2.10}
  are satisfied.  Then $\widetilde{G}^* g_{\SL(N, \R)/\SO(N)}$ is the
  pullback (from $B$ to $\widetilde{B}$) of a function times $g_B$.
\end{proposition}
\begin{proof}
  Using the conclusion of Proposition \ref{prop2.10}, the third equation of (\ref{2.3})
  becomes
  \begin{equation} \label{2.18}
    G^{IJ} G_{JK,\alpha} G^{KL} G_{LI,\beta} = 4 R_{\alpha \beta} =
    2Rg_{\alpha \beta}.
  \end{equation}
  The left-hand side of (\ref{2.18}) is the local expression for
  $G^* g_{\SL(N, \R)/\SO(N)}$ \cite[(4.16)]{Lott (2007)}. The
  proposition follows.
\end{proof}

\subsection{One dimensional base} \label{subsect2.5}

Returning to the equations (\ref{2.3}), suppose that $n=1$. We only
consider the case $\lambda = 0$. Automatically, $F=0$. We give
$B$ a unit speed parametrization $s$.

\begin{proposition} \label{prop2.19}
  Either $G$ is constant or, up to a change of basis of $e$ and an
  isometric reparametrization of $s$, we have $G(s) = s^A$ where
  $A$ is a symmetric $N \times $N matrix satisfying $\Tr(A) = 2$ and
  $\Tr(A^2) = 4$.
\end{proposition}
\begin{proof}
  From (\ref{2.9}), we have $\frac{d^2}{ds^2} \sqrt{\det G} = 0$. Then
  either $\det G$ is constant or, after an
  isometric reparametrization of $s$, we can write
  $\det G = as^2$ for some $a > 0$, with $s \in (c_1,c_2) \subset (0, \infty)$.

  If $\det G$ is constant then the first equation of (\ref{2.3}) gives
  the matrix equation
  \begin{equation} \label{2.20}
G_{ss} - G_s G^{-1} G_s = 0
\end{equation}
and hence
\begin{equation} \label{2.21}
\Tr(G^{-1} G_{ss}) - \Tr (G^{-1} G_s G^{-1} G_s) = 0,
\end{equation}
while the third equation of (\ref{2.3}) gives
\begin{equation} \label{2.22}
\Tr(G^{-1} G_{ss}) - \frac12 \Tr (G^{-1} G_s G^{-1} G_s) = 0.
\end{equation}
Hence $\Tr (G^{-1} G_s G^{-1} G_s) = 0$, or
$\Tr \left(
\left( G^{- \: \frac12} G_s  G^{- \: \frac12} \right)^2
\right) = 0$.
This implies that $G_s = 0$, so $G$ is constant.

If $\det G = as^2$ then the first equation of (\ref{2.3}) becomes
\begin{equation} \label{2.23}
  G_{ss} + \frac{1}{s} G_s - G_s G^{-1} G_s = 0,
\end{equation}
or $\partial_s (s G^{-1} G_s) = 0$. Thus $G^{-1} G_s = \frac{A}{s}$ for
some matrix $A$. By a linear change of basis of $\{x^I\}_{I=1}^N$,
we can assume that $G(1) = \Id$. Then $G(s) = s^A$. As $G(s)$ is symmetric,
the matrix
$A$ must also be symmetric. As $\det G = a s^2$, we must have $a=1$ and
$\Tr(A) = 2$. 
The third equation in (\ref{2.3}) again becomes (\ref{2.22}), which now implies that
$\Tr(A^2-A) - \frac12 \Tr(A^2) = 0$. Hence
$\Tr(A^2) = 2 \Tr(A) = 4$.
\end{proof}

\section{Collapsing of almost Ricci-flat $4$-manifolds} \label{sect3}

In this section we prove Theorem \ref{thm1.1} in the setting of
$4$-manifolds whose Ricci curvature goes to zero relative to the
volume, and that are locally volume collapsed. Subsection \ref{subsect3.1}
has a review of some of the results of \cite{Cheeger-Tian (2006)} and
their consequences.  Subsections \ref{subsect3.2}, \ref{subsect3.3},
\ref{subsect3.4} and \ref{subsect3.5} give the proof of Theorem
\ref{thm1.1} when the limit space has dimension
three, two, one and zero, respectively.

\subsection{General convergence arguments} \label{subsect3.1}

We 
consider four dimensional compact connected orientable
Riemannian manifolds that have Ricci curvature going to zero,
relative to the volume. This is more general than the
setup of the introduction.
In order to prove properties of such manifolds by
contradiction, one considers sequences
$\{ (M_i, \overline{g}_i )\}_{i=1}^\infty$ where each $M_i$ is a compact
connected orientable
four dimensional manifold and 
$\lim_{i \rightarrow \infty} \| \Ric(M_i, \overline{g}_i) \|_\infty \cdot
\vol(M_i, \overline{g}_i)^{\frac12}
= 0$.

\begin{example}
  If $M$ is the underlying $4$-manifold of a complex
  elliptic surface then LeBrun showed that $M$ admits a sequence of
  Riemannian metrics $\{\overline{g}_i\}_{i=1}^\infty$ with
  $\lim_{i \rightarrow \infty} \| \Ric(M, \overline{g}_i) \|_\infty \cdot
  \vol(M, \overline{g}_i)^{\frac12} = 0$ if and only if $M$ is relatively
  minimal, i.e. has no smooth rational $(-1)$-curves in the fibers
  \cite[Theorem 4]{LeBrun (1999)}.
\end{example}

After rescaling, we can assume that
$\lim_{i \rightarrow \infty} \| \Ric(M_i, \overline{g}_i) \|_\infty = 0$ and
$\vol(M_i, \overline{g}_i) \le V$ for all $i$, for some $V < \infty$.
We first address the noncollapsing case.
Suppose that for some $s,v > 0$, after passing to a subsequence
there are points $m_i \in M_i$ so that 
$\vol(B_s(m_i)) \ge v s^4$ for all $i$.
After passing to a further subsequence,
there is a pointed Gromov-Hausdorff limit
$\lim_{i \rightarrow \infty} (M_i, \overline{g}_i, m_i) = (X, d_X, x_\infty)$,
where $X$ is a complete locally compact metric space whose Hausdorff dimension
is four.

If there is a uniform
upper bound on the $L^2$-norms of the curvatures of
$\{ (M_i, \overline{g}_i )\}_{i=1}^\infty$
then
$X$ is a four dimensional Ricci-flat orbifold
\cite{Anderson (1990),Bando-Kasue-Nakajima (1989),Tian (1990)}.
Such an $L^2$-curvature bound is guaranteed 
if there is a uniform upper bound on the
Euler characteristics of the $M_i$'s
\cite[Remark 1.4]{Cheeger-Tian (2006)}; this in turn is guaranteed if
there is a uniform upper
diameter bound
\cite{Cheeger-Naber (2015)}.

The subject of this paper is rather the collapsing case when
\begin{equation} \label{neweqn}
  \lim_{i \rightarrow \infty} s^{-4} \sup_{m \in M_i} \vol(B_s(m)) = 0
\end{equation}
for each $s > 0$.
To apply the results of \cite{Cheeger-Tian (2006)}, we need a uniform
  upper bound on the $L^2$-norms of the curvatures
  of the $\{ (M_i, \overline{g}_i )\}_{i=1}^\infty$. Again,
 it  suffices to have a uniform upper bound on the Euler characteristics
  of the $M_i$'s.

From (\ref{neweqn}),
for any sequence $\{s_i\}_{i=1}^\infty$ of positive numbers converging to
zero, after passing to a further subsequence of
$\{ (M_i, \overline{g}_i )\}_{i=1}^\infty$ we can assume that
$s_i^{-4} \vol(B_{s_i}(m)) \le \frac{1}{i}$ for all
$m \in M_i$. From 
\cite[Theorem 0.1 and Remark 5.11]{Cheeger-Tian (2006)},
there is some positive integer
${\mathcal N}$
so that for all large $i$, there are
points $\{p_{i,j}\}_{j=1}^{{\mathcal B}_i}$ in $M_i$, with
${\mathcal B}_i \le {\mathcal N}$,
such that
\begin{equation}
  \int_{M_i - \bigcup_{j=1}^{{\mathcal B}_i} B_{p_{i,j}}(s_i)} |\Rm|^2 \: \dvol_{M_i}
  \le \const i^{-1}.
  \end{equation}

Choose basepoints $m_i \in M_i$. 
After passing to a further subsequence, we can assume
that $\lim_{i \rightarrow \infty} (M_i, \overline{g}_i, m_i) =
(X, d_X, x_\infty)$ in the pointed Gromov-Hausdorff topology,
where $X$ is a complete locally compact metric space whose Hausdorff
dimension is less than four. We can also assume that
the ${\mathcal B}_i$'s are all the same number, say ${\mathcal B}$.
Then we can assume that for each $j \in \{1, \ldots, {\mathcal B} \}$,
either $\lim_{i \rightarrow \infty} p_{i,j} = x_j$ for some
$x_j \in X$ or $\lim_{i \rightarrow \infty} d_{M_i}(m_i, p_{i,j}) = \infty$.
After removing repetitions, let $\{x_j\}_{j=1}^C \subset X$ be
the limits.

From \cite[Theorem 0.8 and Remark 8.22]{Cheeger-Tian (2006)},
for any compact subset
$K$ of $X - \{x_1, \ldots, x_C\}$, there  is some
$\epsilon_{K} > 0$ so that for any $q \in [1, \infty)$ and for large $i$,
on the subset of $(M_i, \overline{g}_i )$ that is $\epsilon_{K}$-close
to $K$ we have uniformly bounded $W^{2,q}$-covering geometry;
see also
\cite[Theorem 1.1]{Li (2010)}.
Hence we can apply techniques from bounded curvature collapse
\cite[Remark 2.7]{Cheeger-Tian (2006)}.
We will use convergence of Riemannian groupoids, as in
\cite[Sections 5.1-5.4]{Lott (2007)}.
Choose $x^\prime \in X - \{x_1, \ldots, x_C\}$.  (The choice of
$x^\prime$ will be irrelevant.)  Choose a sequence
$m_i^\prime \in M_i$ with $\lim_{i \rightarrow \infty} m_i^\prime =
x^\prime$.
Using an exhaustion of $X - \{x_1, \ldots, x_C\}$ by precompact open
sets containing $x^\prime$,
after passing to a subsequence,
we can assume that
$\lim_{i \rightarrow \infty}
\left( M_i - \bigcup_{j=1}^{{\mathcal B}_i}
B_{p_{i,j}}(s_i),\overline{g}_i, m_i^\prime \right) =
({\mathcal X}, g_{{\mathcal X}}, x^\prime)$, where
$({\mathcal X}, g_{\mathcal X}, x^\prime)$ is a
four dimensional closed
Hausdorff pointed Riemannian groupoid whose orbit space ${\mathcal O}$ is
$X - \{x_1, \ldots, x_C\}$, and we can think of $x^\prime$ as an orbit.
Taking the metric completion of ${\mathcal O}$ amounts to
adding a finite number of points.
The unit space of ${\mathcal X}$ carries a structure sheaf
$\underline{\mathfrak g}$ of
nilpotent Lie algebras, which acts on the unit space
by local Killing vector fields.
The local Killing vector fields do not simultaneously vanish at any point in
the unit space.

The metric convergence to $({\mathcal X}, g_{{\mathcal X}})$
is in the weak $W^{2,q}$-topology,
for any $1 \le q < \infty$. Hence the metric
$g_{{\mathcal X}}$ on the unit space of
${\mathcal X}$ is Ricci-flat. As the limit is constructed using harmonic
coordinates, it follows that the metric on the unit space is smooth.
(If each $M_i$ is Ricci-flat then the convergence is $C^\infty_{loc}$.)

If ${\mathcal X}$ has trivial isotropy
groups then the orbit space ${\mathcal O}$ is smooth.
Then for any precompact open subset $U \subset {\mathcal O}$ and for large $i$,
there is a subset of the ``regular'' region
$M_i - \bigcup_{j=1}^{{\mathcal B}_i} B_{p_{i,j}}(s_i)$
that is the total space
of a fiber bundle over $U$, with infranil fibers.
We are interested in the abelian case with $N$-torus fibers.
In this case, the underlying groupoid of the limit 
${\mathcal X}$ can be described in terms of the following transition
maps. Let $\{U_a\}_{a \in {\mathcal A}}$ be a
covering of ${\mathcal O}$ by
contractible open sets. For $a, b \in {\mathcal A}$, the transition
map $\phi_{ab}$ is a smooth map from $U_a \cap U_b$ to
$\frac{\R^N}{\R^N_\delta} \rtimes \GL(N, \R)_\delta$, where the
$\delta$-subscript denotes the discrete topology
\cite[(5.3)]{Lott (2010)}.
(Compare with the transition maps for a principal $T^N$-bundle,
which take value in $\frac{\R^N}{\Z^N}$, and the transition maps for
a twisted principal $T^N$-bundle, which take value in
$\frac{\R^N}{\Z^N} \rtimes \GL(N, \Z)$.)
That is,
$\phi_{ab}$ is represented by a pair
$(f_{ab}, \gamma_{ab})$ where $f_{ab} : U_a \cap U_b \rightarrow \R^N$
is a smooth map
and $\gamma_{ab} \in \GL(N, \R)$.
Two pairs
$(f_{ab}, \gamma_{ab})$ and $(f^\prime_{ab}, \gamma^\prime_{ab})$ are
equivalent if $\gamma_{ab} = \gamma^\prime_{ab}$ and
$f_{ab} = f^\prime_{ab} + v_{ab}$ for some constant vector
$v_{ab} \in \R^N$. The $\{\phi_{ab}\}_{a,b \in {\mathcal A}}$
have to satisfy
the cocycle condition.
Up to equivalence, such structures on ${\mathcal O}$ are classified by
a set $\HH^1({\mathcal O}, \underline{\mathcal E})$ that fits into
an exact sequence of pointed sets \cite[(5.7)]{Lott (2010)}
\begin{equation}
  \GL(N, \R) \longrightarrow
  \HH^2({\mathcal O}; \R^N) \longrightarrow
  \HH^1({\mathcal O}, \underline{\mathcal E})
  \longrightarrow
  \HH^1({\mathcal O}, \GL(N, \R)_{\delta}).
\end{equation}
Here $\GL(N, \R)$ acts on $\HH^2({\mathcal O}; \R^N)$.
The set $\HH^1({\mathcal O}, \GL(N, \R)_{\delta})$ classifies the
flat $\R^N$-vector bundles on ${\mathcal O}$.

A Riemannian metric on the groupoid has the following description.
For each $a \in {\mathcal A}$,
let $G_a$ be a smooth map from $U_a$ to the positive-definite symmetric
$(N \times N)$-matrices, let $A_a$ be an $\R^N$-valued $1$-form on $U_a$
and let $g_a$ be a Riemannian metric on $U_a$. Then the
triples $\{(G_a, A_a, g_a)\}_{a \in {\mathcal A}}$ define a
Riemannian metric on the groupoid if for each
$a,b \in {\mathcal A}$, on $U_a \cap U_b$ we have $(G_b, A_b, g_b) =
(\gamma_{ab} G_a \gamma_{ab}^{T}, \gamma_{ab}(A_a + df_{ab}), g_a)$.
(Compare with (\ref{2.1}).)

We claim that in the collapsing case, the flat $\R^N$-vector bundle
on ${\mathcal O}$
corresponding to $\{\gamma_{ab}\}_{a,b \in {\mathcal A}}$
has holonomy in $\det^{-1}(\pm 1) \subset \GL(N, \R)$.
The reason is that we are considering the case when for any
precompact open set $U \subset {\mathcal O}$ and
for large $i$, part of the regular region of $M_i$ is the total space
of a $T^N$-bundle over $U$.
Given such an $i$,
there is a flat $\R^N$-bundle on $U$, whose fiber over
a point in $U$ is the first cohomology of the torus fiber over
the point. As $\HH^N(T^N; \Z) \cong \Z$, the holonomy of this flat
$\R^N$-bundle lies in $\det^{-1}(\pm 1) \subset \GL(N, \R)$. Passing
to the limit as $i \rightarrow \infty$, the holonomy
will still lie in $\det^{-1}(\pm 1) \subset \GL(N, \R)$.
This is true for any such $U$.

Thus we can apply the results of Section \ref{sect2} about
Ricci-flatness to ${\mathcal X}$.
If ${\mathcal X}$ has finite isotropy
groups then there is a similar statement,
with ${\mathcal O}$ becoming an
orbifold.

Let $B$ be a connected component of 
$X - \{x_1, \ldots, x_C\}$. We replace ${\mathcal X}$ by its
subgroupoid consisting of the
orbits corresponding to points in $B$. In the next four subsections
we prove the properties of ${\mathcal X}$ asserted in
Theorem \ref{thm1.1}.

Theorem \ref{thm1.1} is stated in the
introduction for a sequence with bounded diameter,
Ricci curvature going to zero and
volume going to zero.  In this case there is a uniform upper
volume bound and (\ref{neweqn}) holds.
Hence the discussion in this subsection applies.

\subsection{Three dimensional limit space} \label{subsect3.2}

Suppose that $\dim(B) = 3$.
Then it is a three-dimensional Riemannian orbifold, as the
groupoid ${\mathcal X}$ must have finite isotropy groups
in this case.

\begin{remark} \label{rem3.3}
The appearance of possible orbifold points in $B$ has
nothing to do with the points in $X$ that could arise as limits
of the curvature blowup regions in the $(M_i, \overline{g}_i)$'s.
These were already removed in forming $B$.
There could be orbifold points in $B$ even if the manifolds
$(M_i, \overline{g}_i)$ are flat.
\end{remark}

The matrix $(G_{IJ})$ is just a function $G$ on $B$.

\begin{proposition} \label{prop3.4}
  If $B$ is parabolic and $G$ is uniformly bounded above then $G$ is a
  constant function, $F = 0$ and $B$ is flat.
\end{proposition}
\begin{proof}
  The first two statements follow from Proposition \ref{prop2.10}, which also
  says that
  $B$ is Ricci-flat. Since it is three dimensional, it is flat.
\end{proof}

\begin{example} \label{ex3.5}
Examples of the hypotheses of Proposition \ref{prop3.4} come from the construction
of collapsing Ricci-flat metrics on $K3$ in
\cite{Foscolo (2016)}. There is an $\Z_2$-action on $S^1$ by
complex conjugation,
and hence on $T^3$, with eight fixed points.
The paper \cite{Foscolo (2016)} constructs sequences of Ricci-flat
metrics on $K3$ that converge to $X = T^3/\Z_2$ in the Gromov-Hausdorff
topology.  The subset $B$ is $X$ minus the eight fixed
points and a certain number of other points, where the number
can be chosen between $0$ and $16$.
We note that $B$ is parabolic.

During the collapse, ALF gravitational instantons of dihedral type
bubble off from the eight fixed points.
\end{example}

\subsection{Two dimensional limit space} \label{subsect3.3}

Suppose that $\dim(B) = 2$.

\subsubsection{Finite isotropy groups} \label{subsubsect3.3.1}

Consider first the case when the groupoid ${\mathcal X}$ has trivial isotropy
groups.  Then $B$ is
a smooth surface. For large $i$,
the corresponding subset of $M_i$ is the total space of a fiber
bundle over $B$. As $M_i$ is orientable, the fibers must be $2$-tori.
We identify $\SL(2, \R)/\SO(2)$ with the
hyperbolic plane $H^2$, the latter having a fixed orientation.

\begin{proposition} \label{prop3.6}
  If  $\det(G)$ is bounded above and
  $B$ is parabolic then
\begin{itemize}
\item $\det(G)$ is constant,
\item $F=0$ and
  \item With the right choice of orientation on
  $\widetilde{B}$, the $\rho$-equivariant map
  $\widetilde{G} : \widetilde{B} \rightarrow H^2$
    is holomorphic.
    \end{itemize}
  \end{proposition}
\begin{proof}
  The first two statements follow from Proposition \ref{prop2.10}.
  From Proposition \ref{prop2.17}, the map $\widetilde{G}$ is conformal.
  Then with the right choice of orientation on $\widetilde{B}$, it
  is holomorphic.
\end{proof}

\begin{remark}
  Examples where $\det(G)$ is bounded above come from collapsing
  K\"ahler manifolds $(M_i,\overline{g}_i)$
  that admit holomorphic fiberings over $X$
  with the generic fiber being a torus.
  Then for each $i$, all of the regular fibers have the same volume.  
  \end{remark}

The notion of a semiflat K\"ahler metric is given in
\cite[Section 3]{Gross-Tosatti-Zhang (2013)},
\cite[Section 3.2]{Hein (2012)}, \cite{Hein-Tosatti (2015)}
and
\cite[Section 3.1]{Tosatti-Zhang (2017)}, among other places.
It is usually considered
for torus fibrations, but also makes sense in our context.

\begin{corollary} \label{cor3.7}
  Under the hypotheses of Proposition \ref{prop3.6}, the metric $g_{\mathcal X}$
  on the unit space $M$ of ${\mathcal X}$
  is a semiflat K\"ahler metric.
\end{corollary}

From (\ref{2.18}), the pullback of the Ricci tensor of $B$, to $\widetilde{B}$,
is $\frac14 \widetilde{G}^* g_{H^2}$.

To be more explicit about the semiflat K\"ahler metric, let $\tau = \tau(z)$
be a holomorphic map from an open set $U$ in $B$ to
the upper half plane. A corresponding map $(G_{IJ})$ from $U$
to symmetric matrices is
\begin{equation}
  G = 
\frac{1}{\Im(\tau)}
\begin{pmatrix}
  1 & \Re(\tau) \\
  \Re(\tau) & |\tau|^2
  \end{pmatrix}.
\end{equation}
The complex structure on $M$ can be defined by local holomorphic coordinates.
One local coordinate
is the pullback of a local holomorphic coordinate on $B$. The
other one is $w = x^1 + \tau x^2$.
The K\"ahler form corresponding to $\overline{g}$ is
\begin{equation}
  \overline{\omega} = \omega_B + \frac{i}{2 \Im(\tau)}
  \left( dw - \frac{w-\overline{w}}{\tau-\overline{\tau}} d\tau \right) \wedge
  \left( d\overline{w} - \frac{w-\overline{w}}{\tau-\overline{\tau}}
  d\overline{\tau} \right),
\end{equation}
where $\omega_B$ is the K\"ahler form for $g$.
If $\phi_B$ is a local K\"ahler potential for $\omega_B$, meaning that
$\omega_B = i \partial \overline{\partial} \phi_B$, then a local
K\"ahler potential for $\overline{\omega}$ is
\begin{equation}
\overline{\phi} = \phi_B - \frac{1}{4\Im(\tau)} (w - \overline{w})^2.
\end{equation}
Equation (\ref{2.18}) becomes the statement that the Ricci form on $B$ is
\begin{equation} \label{form}
  \Ric_B = \frac{i}{4\Im(\tau)^2} d\tau \wedge d\overline{\tau} =
 \frac{i|\tau^\prime(z)|^2}{4\Im(\tau)^2} dz \wedge d\overline{z}.  
\end{equation}

\begin{remark} \label{rem3.9}
  Although the metric completion of $B$ amounts to adding a finite number
  of points, it does not follow from this that $B$ is parabolic. For example,
  the domain $U = \{z \in \C : 1 < |z| < 2\}$ is nonparabolic.
  One can construct a metric $e^{2 \phi} g_{Eucl}$ on $U$ so that the
  metric completion consists of adding two points.
\end{remark}

\begin{example} \label{ex3.8}
  The paper \cite{Gross-Wilson (2000)} considers a collapsing sequence of Ricci-flat
  metrics on $K3$, for which $B$ is $S^2$ minus 24 points. In this case, $B$
  is parabolic.  The semiflat metric is described in
  \cite[Example 2.2]{Gross-Wilson (2000)}.

  To briefly summarize the geometry of the collapse, as taken from
  \cite{Gross-Wilson (2000)},
  during the collapse there are 24 Taub-NUT gravitational instantons
  that bubble
  off. A Taub-NUT gravitational instanton is of ALF type, i.e. has
  cubic volume growth.  It may not be evident how truncated ALF instantons
  get attached in the collapsing limit
  to the semiflat metric on a $T^2$-fibration over the
  complement of 24 balls in $S^2$, since there seems to be a discrepancy
  in the limiting dimensions (3 vs. 2).  This is the role of the
  (incomplete) Ricci-flat Ooguri-Vafa
  metric, which provides the approximate geometry over a ball around any of
  the 24 points in $S^2$, for
  the collapsing $K3$ manifold.
  The Ooguri-Vafa manifold contains an approximate (truncated and rescaled)
  Taub-NUT metric.  In effect, the Ooguri-Vafa manifold
  gives a Ricci-flat transition region as a cobordism between the boundary
  $3$-sphere of the truncated Taub-NUT metric, with the Hopf $S^1$-action,
  and the $\Nil$-manifold that lives over the boundary of the ball and
  has a twisted $T^2$-action.  This transition region carries a mixed
  $F$-structure in the sense of \cite{Cheeger-Gromov (1986)}.
  \end{example}

Now consider the case when isotropy groups are finite.
Then the orbit space
$B$ is an orbifold.  The results of this subsubsection
extend to the orbifold setting.

\subsubsection{Infinite isotropy groups} \label{subsubsect3.3.2}
Suppose now that some points in the unit space
of ${\mathcal X}$ have
isotropy group isomorphic to $\SO(2)$, and the other isotropy groups
are trivial. Then $B$ is a surface with boundary.
Consider the interior of $B$, i.e.  the subset of $B$ corresponding
to points in the unit space with trivial isotropy group. 

\begin{proposition} \label{prop3.10}
  $\det(G)$ extends continuously to be zero on
  $\partial B$.
  \end{proposition}
\begin{proof}
  Let $m$ be a point in the unit space of ${\mathcal X}$ with isotropy
  group $\SO(2)$. From the slice theorem, a neighborhood of $m$ in
  the unit space is equivariantly diffeomorphic to a neighborhood of
  the origin in $\R \times (\R \times \C)$. Here the
  group action is
  on the $\R \times \C$ factor, with translations of the $\R$-term
  and rotations of the $\C$-term.  The points
  in $\R \times (\R \times \C)$ with trivial isotropy
  are $\R \times (\R \times \C^*)$.
  Their quotient by the group action
  is $\R \times (0, \infty)$, from which one obtains local coordinates for
  the part of the interior of $B$ approaching a boundary point.
  
The metric $\overline{g}$ is smooth in the coordinates given by
$\R \times (\R \times \C)$. To describe $\overline{g}$ in terms of the
setup of Subsection \ref{subsect2.1}, we use
the local parametrization of $\R \times (\R \times \C^*)$ by
$(b^1, b^2, x^1, x^2) \rightarrow (b^1, x^1, b^2 \cos(x^2), b^2 \sin(x^2))$,
where $b^2 > 0$. Then for $b^2$ small, the matrix $G$ is asymptotic to
$
\begin{pmatrix}
  1 & 0 \\
  0 & (b^2)^2
\end{pmatrix}
$. In particular, $\lim_{b^2 \rightarrow 0} \det(G) = 0$, showing that
$\det(G)$ vanishes on $\partial B$.
\end{proof}

Let $\Sigma$ be a two dimensional connected Riemannian manifold
with nonempty boundary.
We say that $\Sigma$
is parabolic if any $C^2$-regular subharmonic function on $\Sigma$
that vanishes on $\partial \Sigma$,
and is bounded above, must be zero.
(The reference \cite{Grigor'yan (1999)} instead imposes Neumann boundary
conditions.)

\begin{proposition} \label{prop3.11}
Suppose that $\det(G)$ is bounded above.
  Suppose that the isotropy groups of points in
  the unit space of ${\mathcal X}$ are trivial or isomorphic to $\SO(2)$,
  with not all of them being trivial. Then $B$ cannot be
  parabolic.
  \end{proposition}
\begin{proof}
  By assumption, $B$ is a smooth surface with nonempty boundary.
  If it is parabolic then as in Proposition \ref{prop3.6}, the function $\det(G)$
  is a nonzero constant.  This contradicts Proposition \ref{prop3.10}.
\end{proof}

Finally, suppose just that not all points in the unit space of
${\mathcal X}$ have a finite isotropy group.  Then
$B$ is a two dimensional orbifold-with-boundary, with a
nonempty boundary.  It makes sense to talk about $B$ being
parabolic.  Proposition \ref{prop3.11} has the following extension.

\begin{proposition} \label{prop3.12}
Suppose that $\det(G)$ is bounded above.
  Suppose that not all points in the unit space of
${\mathcal X}$ have a finite isotropy group.
Then $B$ cannot be
  parabolic.
  \end{proposition}

\begin{example}
Suppose that $\det(G)$ is bounded above and
$B$ is compact. Then it is a compact orbifold-with-boundary and hence
is parabolic.
From Propositions \ref{prop3.6} and \ref{prop3.12},
and (\ref{form}),
$B$ is a boundaryless orbifold with nonnegative scalar curvature. Hence
the orbifold universal cover $\widetilde{B}$ is either isometric to $\R^2$ or
has underlying topological space
$S^2$ with zero, one or two orbifold singular points.
In either case, the holomorphic map
$\widetilde{G} : \widetilde{B} \rightarrow H^2$ must be
a constant map.  Then (\ref{form}) implies that $B$ is flat,
so $\widetilde{B} = \R^2$.
In conclusion, the Riemannian groupoid ${\mathcal X}$ is flat.
\end{example}

\subsection{One dimensional limit space} \label{subsect3.4}

Suppose that $\dim(B) = 1$.
Then $B$ is a circle or an interval (possibly open, closed or half-closed).

The sheaf $\underline{\mathfrak g}$ is a sheaf of three dimensional nilpotent
Lie algebras.
Over the interior of $B$,
for large $i$, the corresponding subset of $M_i$ is
a fiber bundle, whose fiber is a three dimensional infranilmanifold.
By \cite[p. 291]{Lott (2002)}, after passing to a finite cover we can
assume that the fiber is a nilmanifold. (This is obvious when $B$ is
an interval.)

Suppose first that the Lie algebra ${\mathfrak g}$ is abelian.
Then the fiber is a $3$-torus.
If $B$ is a circle then from Proposition \ref{prop2.19},
$(G_{IJ})$ is locally constant and the monodromy of $e$ around $B$ is
orthogonal. In particular, $\overline{g}$ is flat.
If $B$ is an interval then from Proposition \ref{prop2.19}, the metric on the
preimage of the interior of $B$, in the unit
space of ${\mathcal X}$, 
is locally isometric to
\begin{equation} \label{3.13}
ds^2 + (dx^1)^2 + (dx^2)^2 + (dx^3)^2
\end{equation}
or
\begin{equation} \label{3.14}
ds^2 + s^{2p_1} (dx^1)^2 +  s^{2p_2} (dx^2)^2 +  s^{2p_3} (dx^3)^2,
\end{equation}
where $p_1 + p_2 + p_3 = p_1^2 + p_2^2 + p_3^2 = 1$.
The metric (\ref{3.14}) is a Riemannian version of the Ricci-flat Kasner
metric from general relativity
\cite[Section 9.1.1]{Ellis-Wainwright (1997)}.

The other possibility for ${\mathfrak g}$ is the three dimensional Heisenberg
Lie algebra. In the Lorentzian case, such Ricci-flat metrics are called
Bianchi type II and are discussed in 
\cite[Section 6.3.2]{Ellis-Wainwright (1997)}. An explicit solution is the
Taub vacuum metric \cite[Section 9.2.1]{Ellis-Wainwright (1997)}.
In the Riemannian case, the corresponding Bianchi II metrics are described in
\cite[Section 4.1]{Dunajski-Tod (2017)} and
\cite[Section 4.2]{Terzis-Christodoulakis (2012)}.
One explicit solution is the
Gibbons-Hawking metric that appears in
\cite[Section 2.2]{Hein-Sun-Viaclovsky-Zhang (2018)}.

\begin{example} \label{addedexample}
  There are Ricci-flat $K3$ manifolds that collapse to a closed interval,
  for which the regular
  regions approach a flat four dimensional geometry over an open interval
  \cite[Section 5]{Chen-Chen (2016)}. During the collapse, an
  ALH gravitational instanton bubbles off from each end of the interval.

  There are also Ricci-flat $K3$ manifolds that collapse to a closed interval,
  with the regular regions collapsing to two open intervals, so that
  over each of these intervals,
  the geometries of the regular
  regions approach a Riemannian Bianchi-II geometry
  \cite{Hein-Sun-Viaclovsky-Zhang (2018)}.
  The construction uses the fact that in the four dimensional case,
  the Ricci-flat metrics constructed
  in \cite[Theorem 4.1]{Tian-Yau (1990)}
  have an asymptotic geometry of Riemannian Bianchi-II
  type.

  We do not know if there are almost Ricci-flat metrics that collapse
  to an interval, for which the regular
  regions approach Riemannian Kasner geometries over open intervals.
\end{example}

\subsection{Zero dimensional limit space} \label{subsect3.5}

If ${\mathcal O}$ is a point then the unit space of ${\mathcal X}$ is
locally homogeneous with a Ricci-flat metric.  Hence it must be flat
\cite{Spiro (1993)}.

\end{document}